\documentclass[12pt]{article}
\usepackage{amsmath, amsthm, amssymb,enumerate,algorithm2e,tikz,caption,bbm}
\usepackage{tikz}
\usepackage{fullpage}
\usepackage{enumitem}
\usepackage{hyperref}
\usepackage{xcolor}
\usepackage{url}

\newtheoremstyle{case}{}{}{\normalfont}{}{shape}{:}{ }{}

\newtheorem{thm}{Theorem}[section]
\newtheorem{lem}[thm]{Lemma}

\newtheorem{cor}[thm]{Corollary}

\newtheorem{claim}[thm]{Claim}

\theoremstyle{definition}

\numberwithin{equation}{section}

\newtheoremstyle{case}{}{}{\normalfont}{}{shape}{\normalfont:}{ }{}

\theoremstyle{case}

\def\comment#1{}

\newcommand{\beq}{\begin{eqnarray*}}
\newcommand{\eeq}{\end{eqnarray*}}

\def\build#1_#2^#3{\mathrel{\mathop{\kern 0pt#1}\limits_{#2}^{#3}}}
\newcommand{\beqs}{\begin{eqnarray}}
\newcommand{\eeqs}{\end{eqnarray}}
\newcommand{\beqsn}{\begin{eqnarray*}}
\newcommand{\eeqsn}{\end{eqnarray*}}

\newcommand{\bN}{\mathbb{N}}

\newcommand{\cA}{\mathcal{A}}

\newcommand{\cF}{\mathcal{F}}

\newcommand{\cH}{\mathcal{H}}

\newcommand{\cS}{\mathcal{S}}

\numberwithin{equation}{section}

\title{Vertex-isoperimetric stability in the hypercube}
\author{Micha{\l} Przykucki\thanks{School of Mathematics, University of Birmingham, Edgbaston, Birmingham, United Kingdom. Supported by the EPSRC grant EP/P026729/1. E-mail: m.j.przykucki@bham.ac.uk.} \and Alexander Roberts\thanks{Mathematical Institute, University of Oxford, Andrew Wiles Building, Radcliffe Observatory Quarter, Woodstock Road, Oxford, United Kingdom. E-mail: robertsa@maths.ox.ac.uk.}}

\begin{document}

\maketitle

\begin{abstract}
Harper's Theorem states that, in a hypercube, among all sets of a given fixed size the Hamming balls have minimal closed neighbourhoods. In this paper we prove a stability-like result for Harper's Theorem: if the closed neighbourhood of a set is close to minimal in the hypercube, then the set must be very close to a Hamming ball around some vertex.
\end{abstract}

\section{Introduction}\label{waffle1}
For all natural numbers $n$, we define the $n$-dimensional hypercube $Q_n = (V,E)$ where $V = \{0,1\}^n$ and $uv \in E$ if the two vertices differ in exactly one co-ordinate. For a vertex $u \in V$ inductively we let $\Gamma^0(u) = \{u\}$, $\Gamma^1(u) = \Gamma(u) = \{v \in V(Q_n): uv \in E(Q_n)\}$, and for $k \geq 2$ we have $\Gamma^{k}(u) = \bigcup_{v \in \Gamma^{k-1}(u)}\Gamma(v) \setminus \Gamma^{k-2}(u)$ (so $\Gamma^k(u)$ is the set of vertices which have shortest path length to $u$ equal to $k$). For a subset of the vertices $U \subseteq V$, we also write $\Gamma(U) = \bigcup_{u \in U} \Gamma(u)$, and we define the \em closed neighbourhood \em of $U$ to be $U \cup \Gamma(U)$, the set of vertices in $U$ together with the neighbourhood of $U$. Note that according to our definition, $\Gamma(U)$ is not necessarily disjoint from $U$; namely, every $u \in U$ with at least one neighbour in $U$, will be contained in $\Gamma(U)$. 

Let $A, B \subseteq [n] = \{1,2,\ldots,n\}$ with $|A| = |B| = r$. We say that $A <_L B$, i.e., that $A$ precedes $B$ in the \emph{lexicographic} (or \emph{lex}) \emph{ordering} on the sets of size $r$, if and only if
\[
\min A \triangle B = \min ((A \cup B) \setminus (A \cap B)) \in A.
\]
Next, let $<_S$ be the ordering of subsets of $[n]$ such that $A <_S B$ if $|A| < |B|$ or if $|A| = |B|$ and $A <_L B$. This is known as the \emph{simplicial ordering}. Since with every vertex $v = (v_1, \ldots, v_n) \in V(Q_n)$ we can naturally associate a set $Z_v = \{i \in [n]: v_i = 1\}$, the orderings $<_L$ and $<_S$ induce orderings on $V(Q_n)$: for $u,w \in V(Q_n)$ we have $u <_L w$ if $Z_u <_L Z_w$, and $u <_S w$ if $Z_u <_S Z_w$. The following well known result of Harper \cite{HARP} (see also \cite[\S16]{bolcomb}) shows that initial segments of $<_S$ have minimal closed neighbourhoods.

\begin{thm}\label{harpermod}
 For each $\ell \in \bN$, let $S_\ell$ be the first $\ell$ elements of $V(Q_n)$ according to $<_S$. If $D \subset V(Q_n)$ with $|D| = \ell$, then
  \[
   |\Gamma(D) \cup D| \ge |\Gamma(S_\ell) \cup S_\ell|.
  \]
\end{thm}

When $\ell = \binom{n}{k}$ and $\sum_{i=0}^{k-1}\binom{n}{i} = o\left(\binom{n}{k}\right),$ $S_{\ell}$ closely resembles a $k$-th neighbourhood (the set of vertices at distance $k$ from a vertex). In this instance, by the well known LYM-inequality (see Lemma \ref{locLYM} to come), the closed neighbourhood of $S_{\ell}$ has size at least
	\begin{align*}
		\left|\Gamma(S_{\ell})\cup S_{\ell}\right| \ge \sum_{i=0}^k \binom{n}{i} + \frac{\left(\ell-\sum_{i=0}^{k-1}\binom{n}{i}\right)}{\binom{n}{k}}\binom{n}{k+1} = \binom{n}{k+1} + O\left(\frac{1}{k} \binom{n}{k}\right).
	\end{align*}

Two questions arise. Firstly, must all sets of order $\binom{n}{k}$ with minimal closed neighbourhood closely resemble a $k$-th neighbourhood of a vertex? Secondly, what happens when a set of size $\binom{n}{k}$ has close to the minimal closed neighbourhood? In this paper we answer the second question through a stability theorem when $k$ is not too large; consequently, our result also answers the first question in the positive. Note that in Theorem~
\ref{thm:f-k-stability} we consider neighbourhoods of sets of vertices rather than closed neighbourhoods, but since these differ by at most $\binom{n}{k}$ vertices this does not change the nature of our result.

\begin{thm}\label{thm:f-k-stability}
Let $\rho$ and $\kappa$ be positive real numbers. Then there exists $n_0 = n_0(\rho,\kappa) \in \bN$ and $\delta = \delta(\rho,\kappa)>0$ such that the following holds: Let $k : \bN \to \bN$ and $p : \bN \to [\rho,\infty)$ be functions such that $k(n) \le \tfrac{\log n}{3\log \log n}$, $\tfrac{k(n)}{p(n)} \leq \kappa$, and $\tfrac{p(n) k(n)^3}{n} \le \delta$. Then for $n \ge n_0$, the following holds: If $A \subseteq V(Q_n)$ with $|A| = \binom{n}{k(n)}$ and $|\Gamma(A)| \le \binom{n}{k(n)+1} + \binom{n}{k(n)}p(n)$, then there exists some $w \in V(Q_n)$ for which we have
  \begin{equation}
   \label{eqn:errorBound}
 |\Gamma^{k(n)}(w) \cap A| \ge \binom{n}{k(n)} - C\binom{n}{k(n)-1}p(n)k(n),
  \end{equation}
where $C = 24 + 33/\rho + 32 \kappa$. 
\end{thm}

Throughout the paper we use the notation $f(n) = O(g(n))$ to mean that there exists some constant $C > 0$ such that $|\tfrac{f(n)}{g(n)}| \leq C$ for all $n$, and $f(n) = o(g(n))$ to say that $\tfrac{f(n)}{g(n)} \to 0$ as $n \rightarrow \infty$. For the ease of notation, we shall often denote $k = k(n)$ and $p = p(n)$.

Let us briefly discuss the sharpness and some limitations of Theorem~\ref{thm:f-k-stability}. Let $A \subseteq V(Q_n)$ be a set of size $|A| = \binom{n}{k(n)}$ satisfying $|\Gamma(A)| \le \binom{n}{k(n)+1} + \binom{n}{k(n)}p(n)$. Then let $w \in V(Q_n)$ be a vertex of the hypercube maximising the value of $|\Gamma^{k(n)}(w) \cap A|$. By~\eqref{eqn:errorBound} we know that at most $C\binom{n}{k(n)-1}p(n)k(n)$ vertices of $A$ lie outside of $\Gamma^{k(n)}(w)$. Can we match this bound? For example, the desired size of $|\Gamma(A)|$ (up to some lower order terms) could be obtained by building $A$ as a disjoint union of $ \binom{n}{k(n)}-\binom{n}{k(n)-1}p(n)$ vertices in $\Gamma^{k(n)}(w)$, together with the $(k(n)-1)$th neighbourhoods of $p(n)$ other vertices in the cube. This example shows that our bound on $|A \setminus \Gamma^{k(n)}(w)|$ is sharp up to an $O(k(n))$ multiplicative term. We believe that at least for $k(n)$ not too large, the $O(k(n))$ is an artefact of our proof. However, it is possible that for $k(n)$ large (possibly larger than the assumptions of our theorem allow) this extra factor in~\eqref{eqn:errorBound} is necessary.

In Theorem~\ref{thm:f-k-stability} we assume that the set $A$ we consider satisfies $|A| = \binom{n}{k}$. However, by the fact that the size of $\Gamma(A)$ cannot decrease when we remove elements from $A$, we can obtain a similar result for sets of size slightly larger than $\binom{n}{k}$, for example, of size $|A| = \sum_{i=0}^k \binom{n}{i}$ when $k$ is not too large. We do this by taking a subset $B \subset A$ of size $)\binom{n}{k}$, applying Theorem~\ref{thm:f-k-stability} to $B$, and then observing that $|\Gamma^{k(n)}(w) \cap A| \ge |\Gamma^{k(n)}(w) \cap B| \geq \binom{n}{k(n)} - C\binom{n}{k(n)-1}p(n)k(n)$. We believe that with very similar methods, results concerning sets of size $\alpha \binom{n}{k}$ might also be derived. However, we anticipate the technical details would be rather tedious.

The strongly related edge-boundary version of the isoperimetric problem (see, e.g., Harper \cite{HARPedge}, Bernstein \cite{bernstein}, and Hart \cite{hart}) has been considered in the stability context by Ellis \cite{ellis}, Ellis, Keller and Lifshitz \cite{EKLedge}, Friedgut \cite{friedgut}, and others.

There are many other fundamental stability-type results in graph theory: for example, the Erd{\H o}s-Simonovits Stability Theorem \cite{erd-sim} states that an $H$-free graph that is close to maximum in size must in fact be close to a Tur\'an graph. The famous Erd{\H o}s-Ko-Rado Theorem \cite{EKR} concerning the maximum size of intersecting set systems has been extended using stability results by, among others, Dinur and Friedgut \cite{DFstability}, Bollob\'as, Narayanan and Raigorodskii \cite{BNRstability}, and Devlin and Kahn \cite{DKstability}.

The stability versions of extremal results can often be applied even more widely that the statements they extend; indeed, the motivation for this work came from the authors' forthcoming paper with Alex Scott \cite{shotgun} on the shotgun reconstruction in the hypercube.

The paper is organised as follows. In Section \ref{prelim} we prove some preparatory lemmas including a tightening of the Local LYM Lemma, and in Section \ref{proofmain} we prove Theorem \ref{thm:f-k-stability}.

We also remark that Peter Keevash and Eoin Long have independently been working on a similar problem \cite{PKEL}. They use very different techniques and their results give weaker bounds for the set-sizes we consider but work for general sized sets and also for much larger sets (i.e., for $k \gg \tfrac {\log n}{3\log \log n}$, although with $p = O(1/k)$).

\section{Preliminaries}\label{prelim}

Given $0 \leq r \leq n$, let $[n]^{(r)}$ be the family of all $r$-element subsets of $[n]$, also called a \emph{layer}. Along with the lex ordering $<_L$, another important ordering in finite set theory is the \emph{colexicographic}, or \emph{colex}, ordering $<_C$ of layers $[n]^{(r)}$. For $A,B \in [n]^{(r)}$ we have $A <_C B$ if $A \neq B$ and
\[
\max A \triangle B = \max ((A \cup B) \setminus (A \cap B)) \in B.
\]
An important fact connecting the orderings $<_L$ and $<_C$ on $[n]^{(r)}$ is that if $\cF$ is the initial segment of $<_L$ on $[n]^{(r)}$ then $\cF^c = \{[n] \setminus A : A \in \cF\}$ is isomorphic to the initial segment of colex on $[n]^{(n-r)}$ (more precisely, it is the initial segment of colex on $[n]^{(n-r)}$ using the ``reversed alphabet'' where $n < n-1 < \ldots < 1$). Indeed, if $|A| = |B| = r$ and $A <_L B$ then by definition we have $\min ((A \cup B) \setminus (A \cap B)) \in A$, which implies that $\min ((A^c \cup B^c) \setminus (A^c \cap B^c)) \in B^c$. Treating the alphabet as ``reversed'' we see that indeed $A^c <_C B^c$.

Let us now fix some more notation that will be used throughout this paper. For $\cF \subseteq [n]^{(r)}$ we write
\[
 \partial(\cF) = \{A \in [n]^{(r-1)} : \exists B \in \cF, A \subseteq B\}
\]
for the \emph{shadow} of $\cF$, and similarly
\[
 \partial^+(\cF)= \{A \in [n]^{(r+1)} : \exists B \in \cF, B \subseteq A\}
\]
for the \emph{upper shadow} of $\cF$.

It will be useful to be able to bound from below the size of the neighbourhood of a subset of $[n]^{(r)}$ by some function of the size of the subset itself. A good starting point for this is the local LYM-inequality  \cite[Ex. 13.31(b)]{LYMstuff}.

\begin{lem}
\label{locLYM}
Let $\cA \subseteq [n]^{(r)}$, then
\begin{equation}
\label{eq:locLYMlower}
\frac{|\partial(\cA)|}{\binom{n}{r-1}} \ge \frac{|\cA|}{\binom{n}{r}},
\end{equation}
and
\begin{equation}
\label{eq:locLYMupper}
\frac{|\partial^+(\cA)|}{\binom{n}{r+1}} \ge \frac{|\cA|}{\binom{n}{r}}.
\end{equation}
\end{lem}

Theorem \ref{harpermod} and Lemma \ref{locLYM} give us the following corollary.
\begin{cor}
\label{cor:harperBound}
Let $k \in \bN$ and let $B \subseteq V(Q_n)$ with $|B| \leq \binom{n}{k}$. Then
\[
 |\Gamma(B)| \geq |B| \frac{n}{k+1} - 2 \binom{n}{k}.
\]
\end{cor}
\begin{proof}
 We have
 \[
  |\Gamma(B)| \geq |B \cup \Gamma(B)| - |B| \geq |B \cup \Gamma(B)| - \binom{n}{k}.
 \]
 Let $\ell = |B|$. By Theorem \ref{harpermod} we can bound further to obtain
 \[
  |B \cup \Gamma(B)| \geq |\Gamma(S_\ell) \cup S_\ell| \geq |\Gamma(S_\ell)| \geq \sum_{i=1}^{k+1} |\Gamma(S_\ell) \cap [n]^{(i)}| \geq \sum_{i=0}^k |\partial^+(S_\ell \cap [n]^{(i)})|.
 \]
 Applying \eqref{eq:locLYMupper} we then have
 \[
  \sum_{i=0}^k |\partial^+(S_\ell \cap [n]^{(i)})| \geq \sum_{i=0}^k |S_\ell \cap [n]^{(i)}| \frac{n-i}{i+1} \geq |B| \frac{n-k}{k+1} \geq |B| \frac{n}{k+1} - \binom{n}{k},
 \]
 completing the proof.
\end{proof}

Unfortunately the well-known inequality \eqref{eq:locLYMupper} is not quite strong enough for our purpose, and so we will need the following result. 

\begin{lem}\label{kruskalkatona}
Let $m,r,i \in \bN$. If $\cF \subseteq [n]^{(r)}$ has order
\begin{equation}
\label{eq:fOrder}
 |\cF| \in \left [ \binom{n}{r} - \binom{n-i+1}{r}+1,\binom{n}{r} - \binom{n-i}{r} \right ],
\end{equation}
then
\begin{equation}
\label{eq:f+NewBound}
  |\partial^+(\cF)| \ge |\cF| \frac{\binom{n}{r+1} - \binom{n-i}{r+1}}{\binom{n}{r} - \binom{n-i}{r}}.
\end{equation}
\end{lem}

We do not claim that Lemma \ref{kruskalkatona} is unknown, but we have been unable to find a reference and so we provide a proof here. The proof uses the following celebrated result of Kruskal and Katona \cite{Kat68,Krusk63}.
 
\begin{thm}\label{thm:kruskat}
Let $\cF \subseteq [n]^{(r)}$ and let $\cA$ be the first $|\cF|$ elements of $[n]^{(r)}$ according to $<_C$. Then $|\partial(\cF)| \ge |\partial(\cA)|$.
\end{thm}

For the ease of reading, for $0 \leq m \leq n$ we shall use the standard notation $[m,n] = \{m, m+1, \ldots, n\}$.

\begin{proof}[Proof of Lemma \ref{kruskalkatona}]
Let $m,r,i \in \bN$ and suppose $\cF \subseteq [n]^{(r)}$ satisfies \eqref{eq:fOrder}. It is easy to see that $\partial^+(\cF) = (\partial(\cF^c))^c$, and so it suffices to estimate $|\partial(\cF^c)|$. By Theorem \ref{thm:kruskat}, the size of the shadow of $\cF^c$ is at least the size of the shadow of the initial segment of size $|\cF|$ in the $<_C$ order on $[n]^{(n-r)}$.

So suppose that $\cH \subset [n]^{(n-r)}$ is an initial segment of $<_C$ order of size as in \eqref{eq:fOrder}. We first want to claim that
\[
 |\cH| = \sum_{j=0}^{i-2} \binom{n-j-1}{r-1} + s,
\]
where $1 \leq s \le \binom{n-i}{r-1}$. Indeed, observe that the first $\binom{n}{r} - \binom{n-i}{r}$ elements in the $<_L$ order on $[n]^{(r)}$ are the sets that are not fully contained in $[i+1,n]$. These can be listed as the $\binom{n-1}{r-1}$ sets that contain $1$, followed by the $\binom{n-2}{r-1}$ sets that contain $2$ but do not contain $1$, etc., followed finally by the $\binom{n-i}{r-1}$ sets $A$ such that $A \cap [i] = \{i\}$. A similar argument holds for the lower bound in \eqref{eq:fOrder}, which proves our claim.

For $j=0,\ldots,i-2$, let
\[
 \cH_j = \left \{A \cup [n+1-j, n] : A \in [n-j-1]^{(n-r-j)} \right \},
\]
so that $|\cH_j| = \binom{n-j-1}{n-r-j} = \binom{n-j-1}{r-1}$. Then $\cH$, being the initial segment of the $<_C$ order on $[n]^{(n-r)}$, can be expressed as the disjoint union $\cH = \bigcup_{j=0}^{i-2} \cH_j \cup \cS$, where
\[
\cS \subset \left \{A \cup [n+2-i, n] : A \in [n-i]^{(n-r-(i-1))} \right \}
\]
has size $s$.  We may then write the shadow of $\cH$ as the disjoint union
\[
 \partial \cH = \bigcup_{j=0}^{i-2} \left (\partial \cH_j \setminus (\partial \cH_0 \cup \ldots \cup \partial \cH_{j-1}) \right) \cup \left (\partial \cS \setminus (\partial \cH_0 \cup \ldots \cup \partial \cH_{i-2}) \right).
\]
	
For each $j$, $\partial \cH_j \setminus (\partial \cH_0 \cup \ldots \cup \partial \cH_{j-1})$ contains exactly the sets of the form $A \cup [n+1-j, n]$ where $A \in [n-j-1]^{(n-r-j-1)}$. Writing $\cS = \{A \cup [n+2-i,n] : A \in \cA\}$ (so $\cA \subseteq [n-i]^{(n-r-(i-1))}$ has $|\cA| = s$) we similarly see that 
\[
 \partial \cS \setminus (\partial \cH_0 \cup \ldots \cup \partial \cH_{i-2}) = \{A \cup [n+2-i,n] : A \in \partial \cA\}.
\]
	
Hence $\partial \cH$ is the disjoint union and consequently
\begin{align*}
| \partial \cH | & = \bigcup_{j=0}^{i-2} | \{A \cup [n+1-j, n] : A \in [n-j-1]^{(n-r-j-1)}\} | \\
 & \qquad \cup | \{A \cup [n+2-i,n] : A \in \partial \cA\} | \\
 & = \sum_{j=0}^{i-2} \binom{n-j-1}{n-r-j-1} + |\partial \cA|.
\end{align*}
Observing that $(n-j-1) - (n-r-j-1) = r$ and applying \eqref{eq:locLYMlower}, we see
\begin{align*}
 |\partial \cH| & \ge \sum_{j=0}^{i-2} \binom{n-j-1}{r} + \frac{n-r-(i-1)}{r}|\cA| \\
 & =   \sum_{j=0}^{i-2} \frac{n-r-j}{r}\binom{n-j-1}{r-1} + \frac{n-r-(i-1)}{r}s.
\end{align*}
If we divide the above expression by $|\cH|$, we can think of this lower bound as a ``weighted average'', with the weights of the elements of $\cH_j$ equal to $\frac{n-r-j}{r}$, and the weights of the elements of $\cS$ equal to $\frac{n-r-(i-1)}{r}$. This last weight is the smallest, hence increasing $s$ only decreases this average. Therefore we get
\begin{align}
\label{eq:f+Weights}
 \frac{|\partial \cH|}{|\cH|} & \ge \frac{\sum_{j=0}^{i-1} \frac{n-r-j}{r}\binom{n-j-1}{r-1}}{\sum_{j=0}^{i-1}\binom{n-j-1}{r-1}} \nonumber \\
 & = \frac{\sum_{j=0}^{i-1} \binom{n-j-1}{r}}{\sum_{j=0}^{i-1} \binom{n-j-1}{r-1}} \\
 & = \frac{\binom{n}{r+1}-\binom{n-i}{r+1}}{\binom{n}{r} - \binom{n-i}{r}}, \nonumber
\end{align}
completing the proof of the lemma.
\end{proof}

\begin{cor}
 \label{cor:f+monotone}
 The sequence $\frac{\binom{n}{r+1} - \binom{n-i}{r+1}}{\binom{n}{r} - \binom{n-i}{r}}$ in \eqref{eq:f+NewBound} is non-increasing in $i$.
\end{cor}
\begin{proof}
 If $i \geq n-r+1$ then $\binom{n-i}{r+1} = \binom{n-i}{r} = 0$ and the sequence stabilises. For $i \leq n-r$, by \eqref{eq:f+Weights} we have
 \[
  \frac{\binom{n}{r+1}-\binom{n-i}{r+1}}{\binom{n}{r} - \binom{n-i}{r}} = \frac{\sum_{j=0}^{i-1} \frac{n-r-j}{r}\binom{n-j-1}{r-1}}{\sum_{j=0}^{i-1}\binom{n-j-1}{r-1}}.
 \]
 If we move from $i$ to $i+1$ on the left-hand side, in the weighted average on the right-hand side we obtain another term $\binom{n-i-1}{r-1}$ with weight $\tfrac{n-r-i}{r}$; this weight is smaller than all the preceding weights and so the average decreases.
\end{proof}

The next lemma somewhat cleans up the multiplicative factor in Lemma \ref{kruskalkatona}.

\begin{lem}\label{kruskbound}
Suppose $\alpha,c \in (0,1)$ are such that $\binom{n}{r} - \binom{\alpha n}{r} = c\binom{n}{r}$. Then
 \[
  \frac{\binom{n}{r+1} - \binom{\alpha n}{r+1}}{\binom{n}{r} - \binom{\alpha n}{r}} \ge \frac{n-r}{r+1} \left( 1 + \frac{1-c}{r} \right).
 \]
\end{lem}

\begin{proof}
Suppose that $\binom{\alpha n}{r} = (1-c)\binom{n}{r}$. Then
\begin{align*}
 (1-c) & = \prod_{i=0}^{r-1} \frac{\alpha n - i}{n-i} \\
 & = \prod_{i=0}^{r-1} \left( \alpha - (1-\alpha)\frac{i}{n-i} \right) \\
 & \ge \prod_{i=0}^{r-1} \left( \alpha - (1-\alpha)\frac{r}{n-r} \right) \\
 & = \left( \frac{\alpha n - r}{n-r} \right)^r.
\end{align*}
Hence we have that $\tfrac{\alpha n - r}{n-r} \le (1-c)^{1/r}$. Thus
\begin{align*}
 \binom{\alpha n}{r+1} & = \frac{\alpha n - r}{r+1}(1-c)\binom{n}{r} \\
 & = (1-c)\frac{\alpha n -r}{n-r}\frac{n-r}{r+1}\binom{n}{r} \\
 & \le (1-c)^{1+1/r}\binom{n}{r+1}.
\end{align*}
We therefore have
\begin{align*}
 \frac{\binom{n}{r+1} - \binom{\alpha n}{r+1}}{\binom{n}{r} - \binom{\alpha n}{r}} & \geq \frac{\left( 1-(1-c)^{1+1/r} \right)\binom{n}{r+1}}{c\binom{n}{r}} \\
 & = \frac{n-r}{r+1}\frac{c + (1-c) \left(1-(1-c)^{1/r} \right)}{c} \\
 & = \frac{n-r}{r+1} \left(1 + \frac{1-c}{c} \left( 1-(1-c)^{1/r} \right) \right).
\end{align*}
A generalisation of Bernoulli's inequality says that if $x \ge -1$ and $t \in [0,1]$, then we have $(1+x)^t \le 1+tx$. Applying this to the above formula with $x=-c$ and $t=1/r$ we obtain
\begin{align*}
 \frac{\binom{n}{r+1} - \binom{\alpha n}{r+1}}{\binom{n}{r} - \binom{\alpha n}{r}} & \ge \frac{n-r}{r+1} \left (1 + \frac{1-c}{c} \cdot \frac{c}{r} \right) = \frac{n-r}{r+1} \left (1+\frac{1-c}{r} \right).
\end{align*}
\end{proof}

In the proof of Theorem \ref{thm:f-k-stability} we first delete sets of vertices with too many unique neighbours. The next lemma will allow us to impose that after this deletion, we get larger and larger layers around vertices in our set.

\begin{lem}\label{expand}
Let $k = o(\log n)$. For sufficiently large $n$ the following holds. Let $J$ be a subset of the hypercube such that for all $S \subseteq J$,
\begin{equation}
\label{eq:fewUniques}
 |\Gamma(S) \setminus \Gamma(J \setminus S)| \le |S|\frac{n}{k+1} \left(1+\frac{1}{8k} \right).
\end{equation}
Then for any vertex $v$ and $j \le 2k$, if $|J \cap \Gamma^{j}(v)| \in [1,\frac{1}{2}\binom{n}{k}]$, then 
\[
 |J \cap \Gamma^{j+2}(v)| \ge \frac{n}{64k^3}|J \cap \Gamma^{j}(v)|.
\]
\end{lem}

\begin{proof}
Without loss of generality, throughout this proof we assume that $v = (0, \ldots, 0)$, so $Z_v = \emptyset$ and for all $j$ we have $\Gamma^j(v) = [n]^{(j)}$. Let $k = o(\log n)$ and let $J$ be a subset of the vertex set of the hypercube such that \eqref{eq:fewUniques} holds for all $S \subseteq J$. The first and most significant step in the proof will be to find a good lower bound on the ratio $|\partial^+(J \cap \Gamma^{j}(v))| / |J \cap \Gamma^{j}(v)|$, arguing according to three different cases. After this bound is obtained, the lemma will follow quite easily.

Assume that we have $j \le 2k$ with $|J \cap \Gamma^{j}(v)| \in [1,\frac{1}{2}\binom{n}{k}]$. If $j \leq k-1$, then we may appeal to \eqref{eq:locLYMupper} to see that for sufficiently large $n$,
\begin{align*}
 \frac{|\partial^+(J \cap \Gamma^{j}(v))|}{|J \cap \Gamma^{j}(v)|} & \ge \frac{n-j}{j+1} \\
 & \ge \frac{n}{k} - 1 \\
 & = \frac{n}{k+1} \left( 1+\frac{1}{k}-\frac{k+1}{n} \right) \\
 & \ge \frac{n}{k+1} \left( 1+\frac{1}{4k} \right).
\end{align*}

Now suppose that $j \ge k$. By Theorem \ref{thm:kruskat} and the relation between the orders $<_C$ and $<_L$, $|\partial^+(J \cap \Gamma^{j}(v))|$ is minimised when $J \cap \Gamma^{j}(v)$ is the initial segment of size $|J \cap \Gamma^{j}(v)|$ in the $<_L$ order on $[n]^{(j)}$.

First suppose that $|J \cap \Gamma^{j}(v)| \le \binom{n-(j+i)}{k-i}$ for some $i \geq 1$. Then all elements of the initial segment of length $|J \cap \Gamma^{j}(v)|$ in the $<_L$ order on $[n]^{(j)}$ contain the set $[j-k+i]$. So remove $[j-k+i]$ from all sets in $J \cap \Gamma^{j}(v)$ and instead work in $[j-k+i+1,n]$. We now have an initial segment of size $|J \cap \Gamma^{j}(v)|$ in the $<_L$ order in $[j-k+i+1,n]^{(k-i)}$ and so, for sufficiently large $n,$ \eqref{eq:locLYMupper}, together with the fact that $j \leq 2k$ and $i \geq 1$, give
\begin{align*}
 \label{eq: lex1}
 |\partial^+(J\cap \Gamma^{j}(v))| & \ge |J \cap \Gamma^{j}(v)|\frac{n-j}{k-i+1} \\
 & \ge |J \cap \Gamma^{j}(v)|\frac{n}{k+1} \left( 1 + \frac{1}{4k} \right).
\end{align*}

Finally let us consider the case when $|J \cap \Gamma^{j}(v)| > \binom{n-(j+1)}{k-1}$. Since $k=o(\log n)$, we have $|J \cap \Gamma^{j}(v)| \le \frac{1}{2}\binom{n}{k} \le \frac{3}{5} \binom{n-j+k}{k}$ for sufficiently large $n$. Therefore we see that all elements of the initial segment of length $|J \cap \Gamma^{j}(v)|$ in the $<_L$ order on $[n]^{(j)}$ contain the set $[j-k]$. Hence remove $[j-k]$ from all sets and instead work in $[j-k+1,n]$. For convenience, we relabel our ground set so that we work with the initial segment of $<_L$ order in $[m]^{(k)}$ where $m=n-j+k$ instead. For $n$ (and so also $m$) large enough we have
\[
\binom{m}{k} - \binom{m(\frac{1}{3})^{1/k}}{k} \geq \binom{m}{k} - \frac{m^k}{3k!} \geq \frac{3}{5} \binom{m}{k} = \frac{3}{5} \binom{n-j+k}{k} \geq |J \cap \Gamma^{j}(v)|.
\]
By Corollary \ref{cor:f+monotone}, we can apply Lemma \ref{kruskalkatona} with $\cF = J \cap \Gamma^{j}(v)$, $n=m$, $n-i = m(\frac{1}{3})^{1/k}$, and $r=k$, to get
\begin{equation}
\label{eqn:upperShadBound}
 |\partial^+(J \cap \Gamma^{j}(v))| \ge |J \cap \Gamma^{j}(v)| \frac{\binom{m}{k+1} - \binom{m(\frac{1}{3})^{1/k}}{k+1}}{\binom{m}{k} - \binom{m(\frac{1}{3})^{1/k}}{k}}.
\end{equation}
(We note that $m(\frac{1}{3})^{1/k}$ should be an integer to apply Lemma \ref{kruskalkatona}. This can be fixed by considering the ceiling of $m(\frac{1}{3})^{1/k}$, but for ease of reading we refrain from doing this.)
Now since $k$ grows sufficiently slowly, for $n$ sufficiently large we have
\begin{align*}
 \binom{m(\frac{1}{3})^{1/k}}{k} & = \frac{m(\frac{1}{3})^{1/k} (m(\frac{1}{3})^{1/k}-1) \ldots (m(\frac{1}{3})^{1/k}-k+1)}{k!} \\ & = \binom{m}{k}\prod_{i=0}^{k-1}\frac{m(\frac{1}{3})^{1/k}-i}{m-i} \\ &\geq \binom{m}{k} \left(\frac{\left(\frac{1}{3}\right)^{1/k}-\frac{k-1}{m}}{1-\frac{k-1}{m}}\right)^k \geq \frac{1}{4}\binom{m}{k}.
\end{align*}

So for $n$ large enough we have $\binom{m}{k} - \binom{m(\frac{1}{3})^{1/k}}{k} \leq \tfrac{3}{4} \binom{m}{k}$ and we can apply Lemma \ref{kruskbound} to \eqref{eqn:upperShadBound} to find
\begin{align*}
 |\partial^+(J \cap \Gamma^{j}(v))| & \ge |J \cap \Gamma^{j}(v)|\frac{m-k}{k+1} \left( 1+\frac{1-\frac{3}{4}}{k} \right) \\
  & \ge |J \cap \Gamma^{j}(v)|\frac{n}{k+1} \left( 1+\frac{1}{4k} \right).
\end{align*}
In all cases, we see that
\begin{equation}
\label{eq:upperShadowSmall}
 |\partial^+(J \cap \Gamma^{j}(v))| \ge |J \cap \Gamma^{j}(v)|\frac{n}{k+1} \left( 1+\frac{1}{4k} \right).
\end{equation}

Since $j \leq 2k$, each vertex in $\Gamma^{j+2}(v)$ is adjacent to at most $2(k+1)$ vertices in $\partial^+(J \cap \Gamma^{j}(v))$. Together with \eqref{eq:upperShadowSmall}, this gives
\begin{align*}
 |\Gamma(J\cap \Gamma^{j}(v)) \setminus \Gamma(J \setminus \Gamma^{j}(v))| & \ge |\partial^+(J \cap \Gamma^{j}(v))| - (2k+2)|J \cap \Gamma^{j+2}(v)| \\
 & \ge |J \cap \Gamma^{j}(v)|\frac{n}{k+1} \left( 1+\frac{1}{4k} \right) - (2k+2)|J \cap \Gamma^{j+2}(v)|.
\end{align*}
On the other hand, by \eqref{eq:fewUniques}, 
\[
 |\Gamma(J\cap \Gamma^{j}(v)) \setminus \Gamma(J \setminus \Gamma^{j}(v))| \le |J \cap \Gamma^{j}(v)|\frac{n}{k+1} \left( 1+\frac{1}{8k} \right).
\]
Together these inequalities give 
\[
 (2k+2)|J \cap \Gamma^{j+2}(v)| \ge |J \cap \Gamma^{j}(v)|\frac{n}{(k+1)8k},
\]
and so $|J \cap \Gamma^{j+2}(v)| \ge \frac{n}{16k(k+1)^2}|J \cap \Gamma^{j}(v)| \ge \frac{n}{64k^3}|J \cap \Gamma^{j}(v)|$.
\end{proof}

\section{Proof of Theorem \ref{thm:f-k-stability}}
\label{proofmain}

In this section we prove Theorem \ref{thm:f-k-stability}. The nature of the proof is much like that of the Erd\H{o}s-Simonovits stability arguments \cite{erd-sim}. Starting with a set $A$ with close to minimal neighbourhood size, we first delete sets of vertices which contribute too many unique neighbours (neighbours unseen by the rest of $A$). We then build up, layer by layer, a rough structure around a vertex of $A$. If $A$ has many vertices in the $j$-th neighbourhood of a vertex $v$, then there must be many vertices of $A$ in $\Gamma^{j+2}(v)$ (else $A \cap \Gamma^j(v)$ has too many unique neighbours). This will mean that for each vertex $v \in A$, there is some $j(v)$ such that almost all of $A$ is contained in $\Gamma^{2j(v)}(v)$, and we then show that $j(v) = k$ for almost all $v \in A$. This means that we can find two vertices $u,v \in A$ at distance $2k$ from one another with $j(u) = j(v) = k$. A pigeonhole argument then reveals a vertex $w$ between $u$ and $v$ for which $A$ is almost entirely contained in $\Gamma^k(w)$.

\begin{proof}[Proof of Theorem \ref{thm:f-k-stability}]
Let $\kappa, \rho > 0$ and let $k: \bN \to \bN$ and $p: \bN \to [\rho,\infty)$ be functions with $k \leq \tfrac{\log n}{3 \log \log n}$, $k \leq \kappa p$, and $pk^3 / n \leq \delta$ for some $\delta > 0$ small to be defined later. Suppose $A \subseteq V(Q_n)$ with $|A| = \binom{n}{k}$ and $|\Gamma(A)| \le \binom{n}{k+1} + \binom{n}{k}p$. For ease of reading, we now state the following two claims here which we will prove later.

\begin{claim}
\label{claim:discard}
There exists $B \subseteq A$ with $|B| \ge \binom{n}{k} - D \binom{n}{k-1}pk$, where $D  = 16 + 32/\rho$, such that for all $S \subseteq B$ we have
\[
 |\Gamma(S) \setminus \Gamma(B \setminus S)| \le |S|\frac{n}{(k+1)} \left( 1+\frac{1}{8k} \right).
\]
\end{claim}

\begin{claim}
\label{claim:cleaning}
Let $B \subseteq A$ be a set which satisfies Claim \ref{claim:discard}. Suppose that there is a vertex $u \in V(Q_n)$ and an integer $\ell \in [k,2k]$ such that $|B \cap \Gamma^\ell(u)| \ge \frac{65 k^3}{n}\binom{n}{k}$. Then
\[
|B \cap \Gamma^\ell(u)| \geq \binom{n}{k} - C\binom{n}{k-1}pk,
\]
where $C = 24 + 33/\rho + 32\kappa.$
\end{claim}

Using Claims \ref{claim:discard} and \ref{claim:cleaning}, we now start by proving the following claim.

\begin{claim}
\label{claim:vertexBigNbhds}
Let $B \subseteq A$ be a set which satisfies Claim \ref{claim:discard}. For all $v \in B$, there exists a $j(v) \le k$ such that $|\Gamma^{2j}(v) \cap B| \ge |B| - C\binom{n}{k-1}pk$.
\end{claim}

\begin{proof}[Proof of Claim \ref{claim:vertexBigNbhds}]
Fix a vertex $v \in B$ and let $j$ be the least integer such that
\[
 |B \cap \Gamma^{2(j+1)}(v)| < \frac{n}{64k^3}|B \cap \Gamma^{2j}(v)|.
\]
If $j \leq k$ then, by Lemma \ref{expand}, $|B \cap \Gamma^{2j}(v)|$ must be at least $\frac{1}{2} \binom{n}{k}$, which means that we must have $2j \geq k$. Since for $n$ large enough we have $\frac{1}{2} \binom{n}{k} \geq \frac{65k^3}{n} \binom{n}{k}$, by Claim \ref{claim:cleaning} we obtain $|B \cap \Gamma^{2j}(v)| \ge \binom{n}{k} - C\binom{n}{k-1}pk$ as desired.

Suppose now that $j \ge k+1$. Since $v \in B$, we have $|B \cap \Gamma^{0}(v)| = |B \cap \{v\}| = 1$. Then, by the choice of $j$, we obtain
\[
 |B \cap \Gamma^{2(k+1)}(v)| \ge \left ( \frac{n}{64k^3} \right )^{k+1}.
\]	
On the other hand,
\[
 |B \cap \Gamma^{2(k+1)}(v)| \le |B| \le \binom{n}{k}.
\]
Since $k \le \frac{\log n}{3 \log \log n}$, we have a contradiction for $n$ sufficiently large, and so $j \le k$. This completes the proof of Claim \ref{claim:vertexBigNbhds}.
\end{proof}

For $j \le k$, let $H(j) = \{v \in B : j(v) = j\}$. Fix $j < k$, and suppose that there are distinct vertices $u,w \in H(j)$ such that $d(u,w) = 2j$. Without loss of generality, we may assume that $Z_u = \emptyset$ and $Z_w = [2j]$. Observe that 
\[
 \Gamma^{2j}(u) \cap \Gamma^{2j}(w) = \{U \cup W : U \in [2j]^{(j)}, W \in [2j+1,n]^{(j)}\}.
\]
The size of this set is clearly $\binom{2j}{j}\binom{n-2j}{j}$. On the other hand
\begin{align*}
 \Gamma^{2j}(u) \cap \Gamma^{2j}(w) & \supseteq \Gamma^{2j}(u) \cap \Gamma^{2j}(w) \cap B \\
 & = B \setminus \left(\left(B \setminus \Gamma^{2j}(w)\right) \cup \left(B \setminus \Gamma^{2j}(u)\right)\right).
\end{align*}
Recall that by the definition of $j = j(u) = j(w)$ we have $|B \setminus \Gamma^{2j}(u)| \le C\binom{n}{k-1}pk$ and $|B \setminus \Gamma^{2j}(w)| \le C\binom{n}{k-1}pk$, therefore
\begin{align*}
 |\Gamma^{2j}(u) \cap \Gamma^{2j}(w)| &\ge \binom{n}{k} - 2C\binom{n}{k-1}pk \\
 &\ge \binom{n}{k}\left(1-3C\delta\right) \ge \frac{1}{2}\binom{n}{k},
\end{align*}
for $\delta$ sufficiently small. Putting these bounds together gives $\binom{2j}{j}\binom{n-2j}{j} \ge \frac{1}{2}\binom{n}{k}$. But $j < k$ and so
\begin{align*}
 \binom{2j}{j}\binom{n-2j}{j} & \leq 4^j\binom{n}{j} \\
 & \leq 4^k \binom{n}{k} \frac{k}{n-k} \\
 & < \frac{1}{2}\binom{n}{k}.
\end{align*}
for $n$ sufficiently large, since $k \le \tfrac{\log n}{3\log \log n}.$ We have a contradiction and so no two vertices from $H(j)$ can be at distance $2j$ from each other.

Since for any $v \in H(j)$ by definition we have $|B \setminus \Gamma^{2j}(v)| \le C\binom{n}{k-1}pk$, and no two vertices from $H(j)$ can be at distance $2j$ from each other, we obtain $|H(j)| \le C\binom{n}{k-1}pk$. Summing over $j < k$ gives
\begin{align*}
|H(k)| & = |B| - \sum_{j=0}^{k-1} |H(j)| \\
 & \geq |B| - C\binom{n}{k-1}pk^2.
\end{align*}
Therefore for a vertex $v \in H(k),$ since $pk^2 \le \delta\tfrac{n}{k},$
\begin{align*}
	\left|\Gamma^{2k}(v) \cap H(k)\right| &\ge |B\cap \Gamma^{2k}(v)| - \left|B\setminus H(k)\right| \\
	&\ge \binom{n}{k} - C\binom{n}{k-1}pk - C\binom{n}{k-1}pk^2 \\
	&\ge \binom{n}{k}\left(1-3C\delta \right),
\end{align*}
for $n$ sufficiently large. This is positive for $\delta$ sufficiently small so that there must exist two vertices in $H(k)$ at distance $2k$ from each other. Let $u,v \in V$ be such vertices and without loss of generality, suppose that $Z_u = \emptyset$ and $Z_v = [2k]$.

Any vertex in $\Gamma^{2k}(u)\cap \Gamma^{2k}(v) \cap B$ must be of the form $X \cup Y$, where $X \in [2k]^{(k)}$ and $Y \in [2k+1,n]^{(k)}$, and so any such vertex must be at distance $k$ from some vertex in $[2k]^{(k)}$. For $w \in [2k]^{(k)}$, let $f(w) = |\{z \in \Gamma^{2k}(u)\cap \Gamma^{2k}(v) \cap B : d(w,z) = k\}|$. Then we have for sufficiently large $n$ and sufficiently small $\delta,$
\begin{align*}
 \sum_{w \in [2k]^{(k)}} f(w) & = |\Gamma^{2k}(u)\cap \Gamma^{2k}(v)\cap B| \\
 & \geq \binom{n}{k} - C\binom{n}{k-1}pk \\
 &\ge \binom{n}{k}\left(1-2C\delta\right) \ge \frac{1}{2}\binom{n}{k}.
\end{align*}
Hence by the pigeonhole principle, there exists a vertex $w \in [2k]^{(k)}$ for which we have
\[
 |\Gamma^k(w) \cap B| \geq \frac{1}{2}\frac{\binom{n}{k}}{\binom{2k}{k}}. 
\]
Recall that $k \le \frac{\log n}{3 \log \log n}$ so that $\binom{2k}{k} \le \tfrac{2n}{65 k^3}$ and so $|\Gamma^k(w) \cap B| \ge \tfrac{65k^3}{n}\binom{n}{k}$ for sufficiently large $n$. By Claim \ref{claim:cleaning} we have $|\Gamma^k(w) \cap B| = \binom{n}{k} -C \binom{n}{k-1}pk$, proving Theorem \ref{thm:f-k-stability}.
\end{proof}

We now complete our argument by proving Claims \ref{claim:discard} and \ref{claim:cleaning}.
\begin{proof}[Proof of Claim \ref{claim:discard}]
Let us run the following algorithm.

\begin{algorithm}[H]
\SetKw{KwFn}{Initialization}
    \KwFn{Set $i=0$, $B_0=A$}\;
    \While{$\exists S \subseteq B_i$ such that $|\Gamma(S) \setminus \Gamma(B_i \setminus S)| > |S|\frac{n}{(k+1)} \left( 1+\frac{1}{8k} \right)$}{
      pick such an $S$\;
      set $i = i+1$\;
      set $L_i = S$\;
      set $B_i = B_{i-1} \setminus S$\;   
    }
\label{discardalgorithm}
\end{algorithm}
Suppose that the algorithm terminates when $i=m$. Suppose that the algorithm terminates when $i=m$. Since the sets $L_1, \ldots, L_m, B_m$ partition $A$, for any $w \in \Gamma(A)$ we either have $w \in \Gamma(B_m)$, or $w \notin \Gamma(B_m)$ and there is some maximum $i$ such that $w \in \Gamma(L_i)$. This gives
\[
 |\Gamma(A)| = \sum_{i=1}^m|\Gamma(L_i) \setminus \Gamma(B_{i-1}\setminus L_i)| + |\Gamma(B_m)|.
\]
Recall that for each $i<m$ we have $|\Gamma(L_i) \setminus \Gamma(B_{i-1} \setminus L_i)| > |L_i|\frac{n}{k+1}(1+\tfrac{1}{8k})$, and so
\[
 |\Gamma(A)| \ge |A \setminus B_m|\frac{n}{k+1} \left( 1+\frac{1}{8k} \right) + |\Gamma(B_m)|.
\]
Corollary \ref{cor:harperBound} gives $|\Gamma(B_m)| \ge |B_m|\frac{n}{k+1} - 2 \binom{n}{k}$. Therefore
\begin{align*}
 |\Gamma(A)| & \geq |A \setminus B_m| \frac{n}{k+1} + |A \setminus B_m|\frac{n}{8k(k+1)} + |B_m| \frac{n}{k+1} - 2 \binom{n}{k} \\
 & = |A| \frac{n}{k+1} + |A \setminus B_m|\frac{n}{8k(k+1)} - 2 \binom{n}{k} \\
 & = \frac{n!}{k!(n-k)!} \frac{n}{k+1} + |A \setminus B_m|\frac{n}{8k(k+1)} - 2 \binom{n}{k} \\
 & \geq \binom{n}{k+1} + |A \setminus B_m|\frac{n}{8k(k+1)} - 2 \binom{n}{k}.
\end{align*}
Since by assumption $|\Gamma(A)| \le \binom{n}{k+1} + \binom{n}{k}p$ and $p \geq \rho$, we obtain
\begin{align*}
 |A \setminus B_m| & \leq \left( \binom{n}{k}p + 2 \binom{n}{k} \right)  \frac{8k(k+1)}{n} \\
 & \leq \left (1+\frac{2}{\rho} \right ) \binom{n}{k} \frac{8pk(k+1)}{n} \\
 & \leq \left (16+ \frac{32}{\rho} \right ) \binom{n}{k}\frac{pk^2}{n-k+1} \\
 & = D \binom{n}{k-1}pk.
\end{align*}
Setting $B = B_m$ we obtain the desired result.
\end{proof}

\begin{proof}[Proof of Claim \ref{claim:cleaning}]
Let $B$ be the set given by Claim \ref{claim:discard} (so $|B| \geq \binom{n}{k} - D \binom{n}{k-1} p k$). Let $v \in V(Q_n)$ be such that for some $\ell \in [k,2k]$ we have $|B \cap \Gamma^{\ell}(v)| \geq \frac{65k^3}{n}\binom{n}{k}$. (Without loss of generality we again assume that $v=(0,\ldots,0)$, so that $Z_v = \emptyset$.) If we also have $|B \cap \Gamma^{\ell}(v)| \leq \frac{1}{2}\binom{n}{k}$ then by Lemma \ref{expand} we have
\[
 |B \cap \Gamma^{\ell+2}(v)| \geq \frac{n}{64k^3} \frac{65k^3}{n}\binom{n}{k} > \binom{n}{k}
\]
which contradicts the fact that $|B| \le \binom{n}{k}$. Therefore we may assume that $|B \cap \Gamma^{\ell}(v)| \ge \frac{1}{2}\binom{n}{k}$ and so $|A \cap \Gamma^{\ell}(v)| \ge \frac{1}{2}\binom{n}{k}$ and $|A \setminus \Gamma^{\ell}(v)| \le \frac{1}{2}\binom{n}{k}$. Recall that $k \geq 1$ and $p \geq \rho$. Since
\[
pk+2 \leq pk(1+2/\rho) < Cpk,
\]
if $|A \cap \Gamma^\ell(v)| \geq \binom{n}{k} - \binom{n}{k-1}(pk+2)$ then we are done. Hence, throughout the proof, we assume $|A \setminus \Gamma^\ell(v)| \ge \binom{n}{k-1}(pk+2)$.

We can bound the size of the neighbourhood of $A$ from below as follows: We count the neighbours of $A \cap \Gamma^{\ell}(v)$ in $\Gamma^{\ell+1}(v)$ (ignoring the neighbours in $\Gamma^{\ell-1}(v)$), and then we add the neighbours of $A \setminus \Gamma^{\ell}(v)$ that are not in $\Gamma^{\ell+1}(v)$. The latter quantity can again be bounded from below by using the fact that any vertex $u$ in $A \setminus \Gamma^{\ell}(v)$ has either $\ell+2$ neighbours in $\Gamma^{\ell+1}(v)$ (if $u \in \Gamma^{\ell+2}(v)$), or otherwise no such neighbours at all. Therefore we have
\begin{equation}
 |\Gamma(A)| \ge |\Gamma(A \cap \Gamma^{\ell}(v))\cap \Gamma^{\ell+1}(v)| + |\Gamma(A \setminus \Gamma^{\ell}(v))| - |A \setminus \Gamma^{\ell}(v)|(\ell+2). \label{count1}
\end{equation}

As we remarked at the beginning of the proof, we may assume $|A \setminus \Gamma^\ell(v)| \geq \binom{n}{k-1}(pk+2)$. By Theorem \ref{harpermod}, $|\Gamma(A\setminus \Gamma^\ell(v))|$ is at least as large as the upper shadow of the first $|A\setminus \Gamma^\ell(v)| - \sum_{i=0}^{k-1}\binom{n}{i}$ elements of $[n]^{(k)}$ according to the $<_L$ order. Write
\begin{equation}
\label{eq:cDefn}
 c\binom{n}{k} = |A \setminus \Gamma^{\ell}(v)|-\sum_{i=0}^{k-1}\binom{n}{i},
\end{equation}
and observe that by the assumption that $|A \setminus \Gamma^{\ell}(v)| \le \frac{1}{2}\binom{n}{k}$ we have $c \leq 1/2$.

Let $\alpha \in (0,1)$ be such that
\[
 c\binom{n}{k} = \binom{n}{k} - \binom{\alpha n}{k}.
\]
Denoting by $W$ the set of the first $c\binom{n}{k}$ elements of $[n]^{(k)}$ according to the $<_L$ order, by Lemma \ref{kruskalkatona} and Corollary \ref{cor:f+monotone} we have
\begin{align*}
 |\Gamma(A \setminus \Gamma^{\ell}(v))| & \ge |\partial^+(W)| \geq |W| \frac{\binom{n}{k+1} - \binom{\alpha n}{k+1}}{\binom{n}{k} - \binom{\alpha n}{k}} = c\binom{n}{k}\frac{\binom{n}{k+1} - \binom{\alpha n}{k+1}}{\binom{n}{k} - \binom{\alpha n}{k}}
\end{align*}
(As in Lemma \ref{expand} we refrain from ensuring things are integer valued for ease of reading.)

Recalling the relation between $\alpha$ and $c$, Lemma \ref{kruskbound} gives
\begin{equation}
\label{eq:finalEstimate1}
 |\Gamma(A \setminus \Gamma^{\ell}(v))| \geq c\binom{n}{k}\frac{n-k}{k+1} \left( 1+\frac{1-c}{k} \right).
\end{equation}
We clearly have
\[
 |\Gamma(A \cap \Gamma^{\ell}(v))\cap \Gamma^{\ell+1}(v)| = |\partial^+(A \cap \Gamma^{\ell}(v))|.
\]
As we mentioned earlier, for a family $\cA \subseteq [n]^{(\ell)}$ we have $\partial^+\cA = (\partial \cA^c)^c$, thus by Theorem \ref{thm:kruskat} the size of the upper shadow of $\cA$ is minimised when $\cA^c$ is isomorphic to the initial segment of colex $<_C$ on $[n]^{(n-\ell)}$, i.e.,  when $\cA$ is isomorphic to the initial segment of lex $<_L$ on $[n]^{(\ell)}$.

Since $p$ is bounded from below by $\rho$, we have
\[
k \leq pk/\rho < Cpk.
\]
Thus, if $|A \cap \Gamma^{\ell}(v)| \geq \binom{n}{k} - \binom{n}{k-1}k$ then again the claim holds and there is nothing to prove. Hence, we may assume that $\tfrac{1}{2} \binom{n}{k} \leq |A \cap \Gamma^{\ell}(v)| \leq \binom{n}{k} - \binom{n}{k-1}k$. Applying the Pascal's rule $k$ times, we have
\begin{align*}
 |A \cap \Gamma^{\ell}(v)| & \leq \binom{n}{k} - \binom{n}{k-1}k \\
 & = \binom{n-1}{k} + \binom{n-1}{k-1} - \binom{n}{k-1}k \\
 & \leq \binom{n-1}{k} - \binom{n}{k-1}(k-1) \leq \ldots \leq \binom{n-k}{k}.
\end{align*}
Recall also that we have $k \leq \ell \leq 2k$. This implies that $\binom{n-k}{k} \leq \binom{n-(\ell-k)}{k}$. Thus every set in the initial segment of size $|A \cap \Gamma^{\ell}(v)|$ of $<_L$ on $[n]^{(\ell)}$ consists of the set $[\ell - k]$ union one of the $\binom{n-(\ell-k)}{k}$ subsets of $[\ell-k+1,n]$ of size $k$. Hence we can again imagine removing $[\ell-k]$ from all sets in our segment and instead working in $[\ell-k+1,n]$. We now have an initial segment of size $|A \cap \Gamma^{\ell}(v)|$ in the $<_L$ order in $[\ell-k+1,n]^{(k)}$ which we denote by $\cH$. Then \eqref{eq:locLYMupper}, together with the fact that $\ell \leq 2k$, gives
\begin{align}
  |\partial^+(A \cap \Gamma^{\ell}(v))| & \ge |\partial^+(\cH)| \nonumber \\
  & \ge |A \cap \Gamma^{\ell}(v)|\frac{n-(\ell-k)-k}{k+1} \nonumber \\
  & = |A \cap \Gamma^{\ell}(v)| \left ( \frac{n-k}{k+1} -  \frac{\ell-k}{k+1}\right ) \nonumber \\
  & \geq |A \cap \Gamma^{\ell}(v)| \frac{n-k}{k+1} -  |A \cap \Gamma^{\ell}(v)| \nonumber \\
  & \geq |A \cap \Gamma^{\ell}(v)|\frac{n-k}{k+1} - \binom{n}{k}. \label{eq:finalEstimate2}
\end{align}
The facts that $|A \setminus \Gamma^{\ell}(v)| \leq \frac{1}{2}\binom{n}{k}$ and $\ell \leq 2k$ imply that
\begin{equation}
\label{eqn:sillyBound}
|A \setminus \Gamma^{\ell}(v)|(\ell+2) \leq \frac{1}{2}\binom{n}{k} (2k+2) \leq 2k\binom{n}{k}.
\end{equation}
Hence we can rewrite~\eqref{count1} using \eqref{eq:finalEstimate1}, \eqref{eq:finalEstimate2}, and \eqref{eqn:sillyBound}, to obtain
\begin{align*}
 |\Gamma(A)| & \geq |A \cap \Gamma^{\ell}(v)|\frac{n-k}{k+1} - \binom{n}{k} + c\binom{n}{k}\frac{n-k}{k+1} \left( 1+\frac{1-c}{k} \right) -2\binom{n}{k}k \\
 & = \left ( |A \cap \Gamma^{\ell}(v)| + c\binom{n}{k} \right ) \frac{n-k}{k+1} + \frac{c(1-c)}{k} \binom{n}{k} \frac{n-k}{k+1} -3\binom{n}{k}k.
\end{align*}
Since we defined $c\binom{n}{k} = |A \setminus \Gamma^{\ell}(v)|-\sum_{i=0}^{k-1}\binom{n}{i}$, and also we have $c \leq 1/2$, we obtain
\begin{align*}
 |\Gamma(A)| & \geq \left(|A| - \sum_{i=0}^{k-1}\binom{n}{i} \right) \frac{n-k}{k+1} + \frac{c}{2k}\binom{n}{k+1} - 3\binom{n}{k}k \\
 & \geq \binom{n}{k+1} - k \binom{n}{k-1} \frac{n-k}{k+1} + \frac{c}{2k}\binom{n}{k+1} -3\binom{n}{k}k \\
 & \geq \binom{n}{k+1} + \frac{c}{2k}\binom{n}{k+1} - 4\binom{n}{k}k.
\end{align*}
Since we assume $|\Gamma(A)| \leq \binom{n}{k+1} + \binom{n}{k}p$, and $k \leq \kappa p$, we must have
\[
 c \le \frac{2k}{\binom{n}{k+1}} \binom{n}{k}(p+4k) \leq \frac{2pk(k+1)(1+4\kappa)}{n-k}.
\]
By the definition of $c$ in \eqref{eq:cDefn}, we then have
	\begin{align*}
		|A \setminus \Gamma^\ell(v)| &= \sum_{i=0}^{k-1}\binom{n}{i} + \frac{2pk(k+1)(1+4\kappa)}{n-k}\binom{n}{k} \\
		& \leq k \binom{n}{k-1} + \frac{8pk^2(1+4\kappa)}{n-k+1}\binom{n}{k} \\
		& = k\binom{n}{k-1} +8pk(1+4\kappa)\binom{n}{k-1} \leq \binom{n}{k-1}pk(8+32\kappa + 1/\rho).
	\end{align*}
and so $|B \setminus \Gamma^{\ell}(v)| \leq (8+32\kappa + 1/\rho)\binom{n}{k-1} pk$. Since $|B| \geq \binom{n}{k} - D \binom{n}{k-1} p k$, we then have 
	\begin{align*}
		|B \cap \Gamma^{\ell}(v)| \geq \binom{n}{k} - \left(D+(8+32\kappa + 1/\rho) \right) \binom{n}{k-1} p k = \binom{n}{k} - C \binom{n}{k-1}.
	\end{align*}
\end{proof}

\noindent
{\bf Acknowledgement} \hspace{.02in}
The authors would like to thank Alex Scott for helpful initial discussions of the problem considered in this paper. During a large part of this project, the first author was affiliated with the Mathematical Institute of the University of Oxford.

\bibliographystyle{plain}

\end{document}